\documentclass[reqno, a4paper]{amsart}

\usepackage{amsmath,amsthm,amsfonts,amssymb,epsfig}
\usepackage{setspace}
\usepackage[latin1]{inputenc}
\usepackage{txfonts}
\usepackage{bbm}
\usepackage{enumerate}

\newtheorem{theorem}{Theorem}
\newtheorem*{claim}{Claim}
\newtheorem{assumption}[theorem]{Assumption}
\newtheorem{corollary}[theorem]{Corollary}
\newtheorem{definition}[theorem]{Definition}
\newtheorem{lemma}[theorem]{Lemma}
\newtheorem{proposition}[theorem]{Proposition}
\newtheorem{remark}[theorem]{Remark}

 \newcommand{\eps}{\varepsilon}
 
 \newcommand{\h}{\mathcal{H}}

 \newcommand{\M}{\mathcal{M}}

\newcommand{\OR}{{\mbox{$p^R$}}}
\newcommand{\OM}{{\mbox{$p^M$}}}

\newcommand{\sint}{\;\begin{picture}(1,1)(0,-3)\circle*{2}\end{picture}\; }

\renewcommand{\phi}{\varphi}

\renewcommand{\epsilon}{\varepsilon}
\renewcommand{\phi}{\varphi}

\newcommand{\E}{\mathbb{E}}
\renewcommand{\P}{\mathbb{P}}
\newcommand{\N}{\mathbb{N}}

\newcommand{\R}{\mathbb{R}}
\newcommand{\BbbR}{\mathbb{R}}
\newcommand{\m}{\mathcal M}

\newcommand{\be}{\begin{equation}}
\newcommand{\ee}{\end{equation}}

\newcommand{\bea}{\begin{eqnarray}}
\newcommand{\bes}{\begin{subequations}}
\newcommand{\ees}{\end{subequations}}
\newcommand{\bgt}{\begin{gather}}
\newcommand{\egt}{\begin{gather}}
\newcommand{\eea}{\end{eqnarray}}
\newcommand{\beaa}{\begin{eqnarray*}}
\newcommand{\eeaa}{\end{eqnarray*}}

\renewcommand\P{\mathcal{P}}

\newcommand{\fourIdx}[5]{%
\setbox1=\hbox{\ensuremath{^{#1}}}%
 \setbox2=\hbox{\ensuremath{_{#2}}}%
 \setbox5=\hbox{\ensuremath{#5}}%
 \hspace{\ifnum\wd1>\wd2\wd1\else\wd2\fi}%
 \ensuremath{\copy5^{\hspace{-\wd1}\hspace{-\wd5}#1\hspace{\wd5}#3}%
 _{\hspace{-\wd2}\hspace{-\wd5}#2\hspace{\wd5}#4}%
 }}

\numberwithin{equation}{section}
\numberwithin{theorem}{section}

\begin{document}
\vspace*{-0.4cm}
\title[Robust FTAP and Super-Replication Theorem]{A Model-free Version of the Fundamental Theorem of Asset Pricing and the Super-Replication Theorem}
\author{B. Acciaio$^{\dagger\ddagger}$}
\author{M. Beiglb\"ock$^{\ddagger}$}
\author{F. Penkner$^{\ddagger}$}
\author{W. Schachermayer$^{\ddagger}$}
\thanks{${}^{\dagger}$ University of Perugia, Department of Economics, Finance and Statistics, Via A. Pascoli 20, I-06123 Perugia\\ 
${}^{\ddagger}$ University of Vienna, Faculty of Mathematics, Nordbergstra\ss{}e 15, A-1090 Wien}
\thanks{The authors thank Marcel Nutz and Mikl\'os R\'asonyi for helpful and relevant comments that lead to an improvement of the paper.}
   



\begin{abstract}
We propose a \emph{Fundamental Theorem of Asset Pricing} and a \emph{Super-Replication Theorem} in a model-independent framework. We prove these theorems in the setting of finite, discrete time and a market consisting of a risky asset $S$ as well as options written on this risky asset. As a technical condition, we assume the existence of a traded option with a super-linearly growing payoff-function, e.g., a power option.
This condition is not needed when sufficiently many vanilla options maturing at the horizon $T$ are traded in the market.
\bigskip

\end{abstract}
\maketitle

\maketitle
\section{Introduction}

We consider a finite, discrete time setting and a market consisting of a collection of options $\varphi_i, i\in I$ written on a risky asset $S$. We allow $I$ to be any set and the $\varphi_i$ any kind of (possibly path-dependent) options written on $S$. In this context we address the following questions:
\begin{itemize}
	\item[(Q1)] Does there exist  an arbitrage opportunity?
	\item[(Q2)] For any additional option written on $S$, what is the range of prices that do not create an arbitrage opportunity?
\end{itemize}
These questions have been widely investigated and exhaustively answered in the classical model-dependent framework, where assumptions are made on the dynamics of the underlying process $S$, see \cite{Sc10,Ca10} and the references therein.

In the recent paper we study these problems without making any model assumption. Instead, we consider the set of all models which are compatible with the prices observed in the market, i.e., we follow the \emph{model-independent approach} to financial mathematics.
A particular case is the situation when one observes the prices of finitely many European call options. This is the setup studied in Davis and Hobson \cite{DaHo07}, where the authors identify three possible cases: \emph{absence of arbitrage}, \emph{model-independent arbitrage} and some weaker form of \emph{model-dependent arbitrage}.
In particular, Davis and Hobson find that the expected dichotomy between the existence of a suitable martingale measure and the existence of a model-independent arbitrage does not hold in this specific setting; there can exist a third possibility in which there exists no suitable martingale measure but only model-dependent arbitrage opportunities (cf.\ \cite[Def.\ 2.3]{DaHo07}) can be constructed. A related notion of \emph{weak arbitrage} is considered by Cox and Ob{\l}{\'o}j \cite{CoOb11a}, where also the notion of \textit{weak free lunch with vanishing risk (WFLVR)} \cite[Def.\ 2.1)]{CoOb11a} is introduced in order to tackle the case of infinitely many given options. In the present paper we consider, possibly infinitely many, general path-dependent options and rule out the possibility of \emph{weak} or \emph{model-dependent arbitrage} by assuming that at least one option with super-linearly growing payoff can be bought in the market. This is the key ingredient to obtain the model-free version of the \emph{
Fundamental Theorem 
of Asset Pricing} given in Theorem~\ref{MFTAP_intro} which provides an answer to question (Q1).

In defining \emph{arbitrage} we follow \cite{DaHo07}, where the concept of \emph{model-independent arbitrage} is introduced in a very natural way, namely via semi-static strategies. A semi-static strategy consists of a static portfolio in finitely many options whose prices are known at time zero, and a dynamic, self-financing strategy in the underlying $S$. We say that a \emph{model-independent arbitrage} exists when there is a semi-static portfolio with zero initial value and with strictly positive value at the terminal date. Strict positivity here pertains to \emph{all} possible scenarios;  there is no a priori reference measure to define a notion of \emph{almost all} scenarios. Pioneering work in this regard was done by Hobson in \cite{Ho98a}; we refer to \cite[Section 2.6]{Ho11} for a detailed account of semi-static strategies and robust hedging. 
Cousot \cite{Co04,Co07}, Buehler \cite{Bu06} and Carr and Madan \cite{CaMa05} consider as given the prices of European call options and give, in different settings, necessary and sufficient conditions for the existence of calibrated arbitrage-free models. Davis, Ob{\l}{\'o}j and Raval~\cite[Theorem 3.6]{DaObRa10} tackle  the case where a finite number of put options plus one additional European option with convex payoff is given; also the relevance for robust super-replication is discussed. In a one-period setting and assuming the prices of finitely many options, Riedel \cite{Ri11} proves a robust Fundamental Theorem of Asset Pricing w.r.t.\ a weak notion of arbitrage.\footnote{Under the assumption of a compact state-space, a Super-Replication Theorem is obtained as a corollary in \cite{Ri11}.}
In continuous time the situation is more delicate; for a discussion in this 
setting we refer to Cox and Ob{\l}{\'o}j~\cite{CoOb11a} 
and Davis, Ob{\l}{\'o}j and Raval~\cite{DaObRa10}.

Heading for a Fundamental Theorem of Asset Pricing, the second issue concerns the pricing measures under consideration. Since we do not assume as given a reference measure, the obvious approach consists in considering as \emph{admissible martingale measures} all probability measures on the path-space $\R^T_+$ which are consistent with the observed option prices and under which the coordinate process is a martingale in its own filtration. In this setup we obtain Theorem~\ref{MFTAP_intro}, which connects the absence of arbitrage with the existence of an admissible pricing measure.

Having discussed this relation, it is natural to address the problem of super-replicating any other option written on $S$. The strategies used for replication again are of the semi-static kind described above. A central question is whether a model-free \emph{Super-Replication Theorem} holds true: 
 given a path-dependent derivative $\Phi$, does the minimal endowment $\OR(\Phi)$ required for \emph{super-replication} equal the \emph{upper martingale price} $\OM(\Phi)$ obtained as the supremum of the expected value over admissible martingale measures? 
In a series of  impressive achievements, Brown, Cox, Davis, Hobson, Klimmek, Madan, Neuberger, Ob{\l}{\'o}j, Pederson, Raval, Rogers,  Wang, Yor, and others \cite{Ro93, Ho98a,BrHoRo01a, HoPe02,MaYo02, CoHoPe08, DaObRa10, CoOb11a,CoOb11b, CoWa11, HoNe12,HoKl12} were able to determine the values $\OR(\Phi)$ and $\OM(\Phi)$ explicitly for specific choices of $\Phi$, showing in particular that they coincide. 
For an overview of the recent achievements we recommend the survey by Hobson \cite{Ho11}. In the approach used by these authors, dominating tools are various Skorokhod-embedding techniques; we refer to the extensive overview given by Ob{\l}{\'o}j in \cite{Ob04}.  
In a discrete time setup, without assuming market-information,
Deparis and Martini \cite{DeMa04} establish the above duality for  $\Phi$ satisfying a particular growth condition.
In a recent article Nutz \cite{Nu13} focuses on optimal super-replication strategies in a (discrete time) setup where super-replication is understood w.r.t.\ a family of probability measures rather than in a path-wise sense.

Recently the super-replication problem in the model-free setting has been addressed via a new connection to the theory of optimal transport; see \cite{GaHeTo11, TaTu12, BeHePe11}. In \cite{GaHeTo11} Galichon, Henry-Labord{\`e}re and Touzi systematically use a controlled stochastic dynamics approach, building on results of Tan and Touzi \cite{TaTu12}. This enables the authors to derive the equation $\OM(\Phi)=\OR(\Phi) $ in the context of the look-back option when the terminal marginal of the underlying is known, recovering in particular results from \cite{Ho98a}. 
This viewpoint is developed further in \cite{ObHeSpTo12} to include market information at intermediate times. 
In a discrete time setup the duality theory of optimal transport can be used to prove $\OR(\Phi)=\OM(\Phi)$ for general
path-dependent $\Phi$ assuming knowledge on the intermediate marginals, see \cite{BeHePe11}. In continuous time (assuming information on the terminal marginal) Dolinsky and Soner \cite{DoSo12} are able to establish the relation $\OR(\Phi)=\OM(\Phi)$ for a large class of path-dependent derivatives. A robust super-replication result in a discrete time setting which also takes proportional transaction costs into account is established in \cite{DoSo13}.
 
In the present article, although inspired by the theory, we do not explicitly use results from optimal mass transport. Instead, we approach the Super-Replication Theorem using the classical route, i.e.\ through the Fundamental Theorem of Asset Pricing (Theorem~\ref{MFTAP_intro}). We obtain the relation  $\OR(\Phi)=\OM(\Phi)$ under fairly general assumptions on the given market-information. In particular we recover the main result of \cite{BeHePe11} as a special case.

\subsection*{Fundamental Theorem of Asset Pricing}
We consider a finite, discrete time setting, with time horizon $T\in\N$, and a risky asset $S=(S_t)_{t=0}^T$, where $S_0$ is a positive real number which denotes the price of $S$ to date. Formally, we take $ S$ to be the canonical process $S_t(x_1, \ldots, x_T)=x_t$ on the path-space $\Omega=\R^{T}_+= [0,\infty)^T$.\footnote{We remark that the results obtained below are also valid in  the case where $S$ is allowed to take values on the whole real line. The proofs carry over to this setup without requiring significant changes.} We also assume that there exists a risk free asset $B=(B_t)_{t=0}^T$ which is normalized to $B_t\equiv 1$.   This setup allows for all possible choices of models since every non-negative stochastic process $S=(S_t)_{t=0}^T$ can be realized using the corresponding measure on the path-space. 

Let $I$ be some index set and $\varphi_i: \BbbR^T_+\to \BbbR$, $i\in I$, the payoff functions of options on the underlying $S$ that can be \emph{bought} on the market at time $t=0$. W.l.o.g.\ we assume that they can be bought at price $0$.
We assume that, if an option $\varphi$ can be both \emph{bought and sold}, then bid and ask prices coincide. In this case we simply include $\pm \varphi$ among the $\varphi_i$.
Consequently the set of admissible measures is defined as
\begin{equation} \label{def:admissible_measures}
\P_{(\varphi_i)_{i\in I}}:=\left\{\pi\in\P(\R^T_+) : \int_{\R^T_+}\varphi_i(x)\,d\pi(x)\leq 0,\, i\in I\right\},
\end{equation}
where $\P(\R^T_+)$ denotes the set of all probability measures on $\R^T_+$.
\begin{definition}[Trading strategies]\label{def.adm}
A trading strategy $\Delta=(\Delta_t)_{t=0}^{T-1} $ consists of Borel measurable  functions  $\Delta_t:\R^t_+\to \R$, where $0\le t <T$. The set of all such strategies will be denoted by $\h$. For the stochastic integral we use the notation
\[
(\Delta \sint x)_T := \sum_{t=0}^{T-1}\Delta_t(x_1,\ldots,x_t)(x_{t+1}-x_t),
\]
so that $(\Delta\sint S)_T$ represents the gains or losses obtained by trading according to $\Delta$. 
\end{definition}

The set of \emph{martingale measures} $\M$ consists of all probabilities on $\R^T_+$ with finite first moment such that the canonical process $S$ is a martingale in its natural filtration. Therefore, the set of admissible martingale measures is given by
\begin{equation} \label{def:admissible_mart_measures}
\M_{(\varphi_i)_{i\in I}}:= \P_{(\varphi_i)_{i\in I}}\cap \M.\end{equation}
 As mentioned above, we define \emph{arbitrage} via semi-static strategies, following \cite[Def.\ 2.1]{DaHo07}.
\begin{definition}[Arbitrage]
There is \emph{model-independent arbitrage} if there exists a trading strategy $\Delta\in \h$  and if there exist constants $a_1,\ldots, a_N\geq 0$ and indices $i_1, \ldots, i_N \in I$ such that
\begin{equation}\label{eq:arb}
f(x_1, \ldots, x_T)=\sum\limits^N_{n=1} a_n\varphi_{i_n}(x_1, \ldots, x_T) + (\Delta \sint x)_T>0
\end{equation}
for all $x_1, \ldots, x_T\in\R_+$.
\end{definition}
We emphasize the fact that the present definition \emph{model-independent} arbitrage requires the strict inequality in \eqref{eq:arb} to hold true \emph{surely}, i.e., on the whole path-space $\R_+^T$.

In the Fundamental Theorem of Asset Pricing given below (Theorem~\ref{MFTAP_intro}) we assume the existence of an option with a super-linearly growing payoff $\phi_0(S)= g(S_T)$ for some convex super-linear function $g:\R_+\to \R$.

\begin{theorem}[FTAP]\label{MFTAP_intro}
Let $\varphi_i, i\in I$ be continuous functions on $\R_+^T$. Let $g:\R_+\to \R$ be a convex super-linear function, i.e., $\lim_{x\to \infty}\frac{g(x)}{x}=\infty$, and assume $\varphi_0$ to be of the form $\varphi_0(S)=g(S_T)$, where we suppose that $0$ is an element of the index-set $I$.
Assume also that 
\begin{equation}\label{cond2}
\lim_{\|x\|\to\infty} \tfrac{\varphi_i(x)^+}{m(x)} <\infty\quad \mbox{ and }\quad \lim_{\|x\|\to\infty} \tfrac{\varphi_i(x)^-}{m(x)} =0,\qquad i\in I,
\end{equation}
where $m(x_1, \ldots, x_T):= \sum_{t=1}^T g(x_t)$.
Then the following are equivalent:
\begin{enumerate}[(i)]
\item\label{it:NA} There is no model-independent arbitrage.
\item\label{it:EMM} $\M_{(\varphi_i)_{i\in I}}\ne\emptyset$.
\end{enumerate}
\end{theorem}
Condition \eqref{cond2} is satisfied, for instance, when the set of the $\varphi_i$ consists of European call options plus one power option $\phi_0$. Note that the second condition in \eqref{cond2} implies that  we cannot sell $\phi_0$ in the market. We can only buy it at a finite, possibly very high, price. Economically, this may be interpreted as the opportunity of an insurance against high values of the stock $S$.

\subsection*{Robust Super-Replication Results} 

As in classical mathematical finance, the Fundamental Theorem of Asset Pricing has a Super-Replication Theorem as immediate corollary. 
\begin{theorem}[Super-Replication]\label{M1thm:superrep}
Let $(\varphi_i)_{i\in I}$ be as in Theorem \ref{MFTAP_intro} and assume that \mbox{$\M_{(\varphi_i)_{i\in I}} \neq \emptyset$}.
Let $\Phi:\R_+^T\to\R$ be u.s.c. and such that
\begin{equation}\label{condwp}
\lim_{\|x\|\to\infty} \tfrac{\Phi(x)^+}{m(x)}=0.
\end{equation}
Then
\begin{align}\label{PrimalRep}
\OM(\Phi):=&\ \sup_{\pi\in\M_{(\varphi_i)_{i\in I}}}\int_{\R^T_+}\Phi(x)\,d\pi(x)\\
=&\ \ \ \inf\ \ \left\{d: \exists a_n\geq 0, \Delta\in \h\, \textrm{ s.t. }\, d+\sum_{n=1}^Na_n\varphi_{i_n}+(\Delta \sint x)_T\geq\Phi\right\}=:\OR(\Phi).\label{DualRep}
\end{align}
In addition, the above supremum is a maximum. 
\end{theorem}

We emphasize that the Super-Replication Theorem perfectly fits the setup of model-independent finance: the financial market provides information about the prices of traded derivatives $\varphi_i, i\in I$. This allows to access the largest reasonable price of the derivative $\Phi$ in two ways.
\begin{enumerate}
\item Following the no-arbitrage pricing paradigm,  one selects a martingale measure $\pi$ which fits to the market prices; the corresponding price for the derivative $\Phi$ equals $\int_{\R^T_+}\Phi(x)\,d\pi(x)$. In general there are infinitely many possible choices for $\pi$ and the robust point of view is to take the martingale measure $\pi$ leading to the largest  value for $\int_{\R^T_+}\Phi(x)\,d\pi(x)$.
 This is  $\OM(\Phi)$ given in \eqref{PrimalRep}. 
\item On the other hand, a robust upper bound to the price of $\Phi$ can be obtained by considering semi-static super-hedges $d+\sum_{n=1}^Na_n\varphi_{i_n}+(\Delta \sint x)_T\geq\Phi$. This approach  was introduced by Hobson (cf.\ \cite{Ho11}) and leads to the 
    value $\OR(\Phi)$ in  \eqref{DualRep}.
\end{enumerate}
Theorem~\ref{M1thm:superrep} asserts that the two approaches are equivalent. (However, while there is always an optimal martingale measure, the existence of an optimal super-hedge is in general not guaranteed.)

The results presented so far required that the market sells a financial derivative $\phi_0(S)=g(S_T)$ where $g$ grows super-linearly. This assumption can be avoided, provided that a sufficient amount of call options written on $S_T$ is traded on the market. 
For instance, it suffices to assume that there is a sequence of strikes  $K_n, n\geq 1, K_n\to \infty$  such that the call options $\psi_{K_n}=(S_T-K_n)_+$ can be bought in the market at price $p_n$, where $p_n\to 0$ as $n\to \infty$. This is spelled out in detail in Corollary \ref{M2thm:superrep} below; in this introductory section we just present a particular consequence.

A prevalent assumption in the theory of model-independent pricing is that the  distribution of $S_T$ can be deduced from market data. 
This is due to the important observation of  Breeden and Litzenberger~\cite{BrLi78} that knowing the law $\nu$ of $S_T$ is equivalent to knowing the prices $p_K$ of $ (S_T-K)_+$ for all strikes $K\geq 0$. The price of an arbitrary European derivative $\phi(S_T)$ is then given by $\E_{\nu}[ \phi(S_T)]=\int_{\R_+}\phi(y)\,d\nu(y)$. We write $\M(\nu)$ for the set of all martingale measures $\pi$  satisfying $S_T(\pi) =\nu$. 
Of course, this set is non-empty if and only if the first moment of $\nu$ exists and equals $S_0$.  
\begin{corollary}[Super-Replication]\label{thm:superrep_term_marg}
Assume that $\nu$ is a probability measure on $\R_+$ with finite first moment and barycenter $S_0$.
Let $\Phi:\R^T_+\to\R$ be u.s.c.\ and linearly bounded from above.
Then
\begin{align*}
\OM(\Phi) :=&\ \sup_{\pi\in\mathcal{M}(\nu)}\left\{\textstyle \int_{\R^T_+} \Phi(x)\,d\pi(x)\right\}\\
=&\ \ \ \inf\ \ \Big\{\textstyle \int_{\R_+}\varphi(y)\,d\nu(y):\varphi\in L^1(\nu), \exists\, \Delta\in\h \ \textrm{s.t.}\ 
\varphi(x_T)+(\Delta\sint x)_T\geq \Phi(x) \Big\}=:\OR(\Phi).
\end{align*}
In addition, the above supremum is a maximum.

More generally these results hold true if there exists a convex super-linear function $\tilde g:\R_+\to\R$ in $L^1(\nu)$ such that 
\begin{equation}
\lim_{\|x\|\to\infty} \tfrac{\Phi(x)^+}{\sum_{t=1}^T\tilde g(x_t)}<\infty.
\end{equation}
\end{corollary}
In the same spirit we also recover  \cite[Theorem 1]{BeHePe11} which corresponds to the Super-Replication Theorem in the particular case where all marginals $S_t\sim\mu_t$, $t=1, \ldots, T$ are known (see Corollary \ref{BHP}). 

Knowing that there is no duality-gap, a natural question is whether the infimum over super-replication strategies is in fact a minimum. In general, this is not the case. In  \cite[Section 4.3]{BeHePe11} a counterexample is given in a setup where $T=2$ and the function $\Phi$ is uniformly bounded. 
As a remedy it may be useful to consider a  \emph{relaxed} notion of super-replication strategies. 
E.g., such ``weak minimizers'' are the critical tool in \cite[Appendix A]{BeJu12}.  

\subsection*{Connection with Martingale Inequalities}
Assume that $\Phi, \phi$ are functions satisfying some proper integrability assumption.
A  path-wise hedging inequality of the form 
\begin{align}\label{08.15} \Phi(x_1,\ldots,x_T) \leq \phi(x_1, \ldots, x_T)+ (\Delta\sint x)_T,\quad x_1, \ldots, x_T\in \R_+,\end{align}
implies that for every martingale $S=(S_t)_{t=1}^T$ we have
$$ \E[\Phi(S_1, \ldots, S_T)] \leq \E[\phi(S_1, \ldots, S_T)].$$
This follows by applying the inequality \eqref{08.15} to the paths of $S$ and taking expectations. 
In short, every  path-wise hedging inequality yields a martingale inequality as a direct consequence. 

Conversely one may ask if a given martingale inequality can be established in this way, i.e.\ as a consequence of a path-wise hedging inequality of the form \eqref{08.15}. In Section \ref{SSSS} below we explain why this can be expected as  a consequence of the Super-Replication Theorem~\ref{M1thm:superrep}.  An early version of this result motivated the path-wise approach to the Doob $L^p$-inequalities given in \cite{AcBePeScTe12}.

\subsection*{Organization of the paper}

 In Sections~\ref{sect:ftap} and \ref{sect:srt} we prove the Fundamental Theorem of Asset Pricing (Theorem~\ref{MFTAP_intro}) and the Super-Replication Theorem (Theorem~\ref{M1thm:superrep}), respectively. 
 Section~\ref{LastOne} collects different super-replication results which do not require the existence of super-linearly growing derivatives in the market. Finally we discuss the relation between robust Super-replication and Martingale Inequalities in Section \ref{SSSS}. 

\section{Fundamental Theorem of Asset Pricing}\label{sect:ftap}

In the definition of model-independent arbitrage we have used trading strategies $\Delta\in \h$ which depend on $S$ measurably but need not be bounded. In particular, $(\Delta\sint S)$  is not necessarily integrable w.r.t.\ a martingale measure $\pi\in \m$. The following remark takes care of this shortcoming.

\begin{remark}\label{rmk.js}
For every $\Delta\in\h$ and $\pi\in\M$, the process $M=(M_t)_{t=0}^T$ defined as
$$M_0:=0,\quad M_t:=(\Delta \sint x)_t,\, t=1,\ldots,T$$
is a discrete-time $\pi$-martingale transform, and hence a $\pi$-local-martingale by Theorem 1 in \cite{JaSh98}. Moreover, if $\int(\Delta \sint x)_T^+ \, d\pi(x)<\infty$ or $\int(\Delta \sint x)_T^- \, d\pi(x)<\infty$, then $M$ is a true $\pi$-martingale, by Theorem 2 in \cite{JaSh98}.
\end{remark}

As a consequence of Remark \ref{rmk.js},
the existence of a martingale measure in $\m_{(\phi_i)_{i\in I}}$ implies that there is no model-independent arbitrage.
\begin{proof}[Proof of Theorem~\ref{MFTAP_intro}, $\eqref{it:EMM}\Rightarrow\eqref{it:NA}$.]
Pick $\pi\in\M_{(\varphi_i)_{i\in I}}$ and assume that there exists $f(x)=\sum_{n=1}^Na_n\varphi_{i_n}(x)+(\Delta\sint x)_T$, where $a_n\geq 0$ and $\Delta\in \h$ such that $f>0$. This gives $\int(\Delta \sint x)_T^-\, d\pi(x)<\infty$, which then, by Remark~\ref{rmk.js}, implies $\sum_{n=1}^Na_n\int\varphi_{i_n}(x)\, d\pi(x)>0$ contradicting the admissibility of $\pi$.
\end{proof}   
In the same fashion, Remark \ref{rmk.js}
yields the ``economically obvious'' inequality $\OM(\Phi) \leq \OR (\Phi)$ in the above super-replication results. 

\medskip

It is natural to ask why we do not only consider bounded strategies. We explain here why this would be too restrictive for our purposes.
For every convex function $g:\R_+\to \R$ and $x_t, x_{t+1}\in\R_+$
 we have\footnote{At the (at most countably many) points where the convex function $g$ is not differentiable, we define $g'$ as its right derivative.} 
\begin{align}\label{ConvexFHedge}
g(x_t)+g'(x_t)(x_{t+1}-x_{t})\leq g(x_{t+1}). 
\end{align} This simple inequality expresses a fact which is widely known in finance under the name of \emph{calendar spread}: a convex derivative written on $S_{t}$ can be super-replicated using the corresponding derivative written on $S_{t+1}$. To incorporate this argument in our path-wise hedging framework, we need to include $\Delta_t(x_1, \ldots, x_t):=g'(x_t)$ in the set of admissible trading strategies.
 
 Indeed, in showing the non trivial implication $\eqref{it:NA}\Rightarrow\eqref{it:EMM}$ in of Theorem~\ref{MFTAP_intro} (and the non-trivial inequality $\OM(\Phi) \geq \OR(\Phi)$ in our Super-Replication Theorems), it is sufficient to use the no arbitrage assumption on a subset of $\h$ which  consists entirely of strategies $\Delta$ such that $(\Delta\sint S)$ is $\pi$-integrable for all $\pi\in \m_{(\phi_i)_{i\in I}}$.  
 
\begin{definition}[$g$-Admissible Strategy]\label{Gdef.adm}  

Let $g\colon \R^+ \to \R$ a convex, superlinear function. A trading strategy $\Delta=(\Delta_t)_{t=0}^{T-1}$ is called \emph{$g$-admissible} if, for $0\leq t\leq T-1$, $\Delta_t:\R^t_+\to \R$ is a continuous function such that, for some $c\in\R_+$,
\begin{align}\label{OurAd}|\Delta_t(x_1, \ldots, x_t)(x_{t+1}-x_t)|\leq c \Big(1\vee \sum_{s=1}^{t+1}g(x_s)\Big).
\end{align}
The set of all $g$-admissible trading strategies is denoted by $\h_g$.
\end{definition}

Trivially we have $\h_g\subseteq \h$.
We briefly comment on the integrability properties of the set $\h_g$.  
Assume that $\pi$ is a  martingale measure on $\R_+^T$ such that $\int g(x_T)\,d\pi(x)< \infty$. By Jensen's inequality we then have 
$\int g(x_t)\,d\pi(x)< \infty$ also for all $t<T$. Thus for $\Delta \in \h_g$, \eqref{OurAd} implies that
$$\int_{\R^T_+} |\Delta_t(x_1, \ldots, x_t)(x_{t+1}-x_t)|\, d\pi(x)<\infty.$$
Disintegrating $\pi$ w.r.t.\ $(x_1, \ldots, x_t)$ it moreover follows that 
$$ \int_{\R^T_+} \Delta_t(x_1, \ldots, x_t)(x_{t+1}-x_t)\, d\pi(x)=0.$$
 Note also that  by \eqref{ConvexFHedge}
 $|g'(x_t)(x_{t+1}-x_{t})| \leq |g(x_t)|+|g(x_{t+1})|$, hence $\Delta_t(x_1, \ldots, x_t):=g'(x_t)$ is $g$-admissible.

In the following proposition we use the notation introduced in \eqref{def:admissible_measures} for the set of admissible measures. Recall that we write $m(x_1, \ldots, x_T)=\sum_{t=1}^T g(x_t)$.
\begin{proposition}\label{Mprop:FTAP1}
Let $\varphi_i:\R^T_+\to\R$, $i=1,\ldots,N$ be continuous functions satisfying
\begin{equation}\label{cond2J}
\lim_{\|x\|\to\infty} \tfrac{\varphi_i(x)^+}{m(x)} <\infty\quad \mbox{and}\quad \lim_{\|x\|\to\infty} \tfrac{\varphi_i(x)^-}{m(x)} =0
\end{equation}
and set $\varphi_{N+1}:= m$ and $\bar{m}:= m \vee 1$.
TFAE:
\begin{enumerate}[(i)]
 \item There is no $f =\sum^{N+1}_{n=1}a_n\varphi_n$ with $a_n\geq 0$ s.t.
  $$
    f(x) >0\quad \mbox{ for all } ~ x\in \R^T_+.
  $$
  \item[(i')] There is no $f=\sum^{N+1}_{n=1}a_n\varphi_n$ with $a_n\geq 0$ s.t.
    $$
      f(x) \geq m(x)\quad \mbox{ for all } ~ x\in \R^T_+.
    $$
  \item $\P_{(\varphi_i)_{i=1}^{N+1}}\neq \emptyset.$
\end{enumerate}
\end{proposition}

\begin{proof}
The only non trivial implication is $(i')\Rightarrow (ii)$:  Consider the Banach space $C_{\bar{m}}^b(\R^T_+)$ of continuous functions $f$ on $\R^T_+$ such that
$$\|f\|_{C_{\bar{m}}^b}= \sup\limits_{x\in\R^T_+} \tfrac{|f(x)|}{\bar{m}(x)} <\infty.$$
The norm is designed in such a way that the multiplication operator $T_{\bar{m}} :C^b_{\bar{m}} (\R^T_+)\to C^b(\R^T_+)$
$$T_{\bar{m}} (f) =\tfrac{f}{\bar{m}}$$
is an isometry, where the Banach space $C^b(\R^T_+)$ of bounded continuous functions $h$ on $\R^T_+$ is endowed with
$$\|h\|_{C^b} =\sup\limits_{x\in\R^T_+} |h(x)| <\infty.$$
Recall that $C^b(\R^T_+)$ may be identified with the space $C(\check{\R^T_+})$ of continuous functions on the Stone-Cech-compactification 
$\check{\R^T_+}$ of $\R^T_+$. 
Hence the dual space of $C^b(\R^T_+)$ can be identified with $\mathcal{M}(\check{\R^T_+}),$ the space of signed Radon measures $\mu$ on $\check{\R^T_+}.$ 
Each $\mu$ can be uniquely decomposed into $\mu=\mu^r+\mu^s$, where the regular part $\mu^r$ is supported by $\R^T_+$ while the singular part $\mu^s$ is supported by 
$\check{\R^T_+} \backslash \R^T_+$.
The bottom line of these considerations is that a continuous linear functional $F$ on $(C^b_{\bar{m}}(\R^T_+),\|\cdot\|_{C^b_{\bar{m}}})$ is given by some 
$\mu=\mu^r+\mu^s\in\mathcal{M}(\check{\R^T_+})$ via
\begin{align}\label{A9}
F(f)& =\int\tfrac{f(x)}{\bar{m}(x)}\,d\mu(x)\\
&=\int \tfrac{f(x)}{\bar{m}(x)} \,d\mu^r(x)+\int\tfrac{f(x)}{\bar{m}(x)}\,d\mu^s(x),\quad \mbox{for} \ f\in C^b_{\bar{m}}(\R^T_+).  \nonumber
\end{align}
Finally observe that the interior of the positive orthant of $C^b_{\bar{m}}(\R^T_+)$ is given by
$$(C^b_{\bar{m}})_{++}(\R^T_+)=\left\{f\in C^b_{\bar{m}} : \inf\limits_{x\in\R^T_+}\tfrac{f(x)}{\bar{m}(x)} >0\right\},$$
as one easily sees from the isometric identification of $C^b_{\bar{m}}(\R^T_+)$ with $C(\check{\R^T_+}).$

\medskip

Turning to the present setting, define $K$ as the compact, convex set in $C^b_{\bar{m}}(\R^T_+)$
$$
K:= \left\{\sum^{N+1}_{n=1} a_n\varphi_n : a_n\geq 0,\sum\limits^{N+1}_{n=1} a_n=1\right\}.
$$
By assumption $(i')$ we have
$$K\cap (C^b_{\bar{m}})_{++}(\R^T_+)=\emptyset,$$
so that we may apply Hahn-Banach to find a linear functional $F\in C^b_{\bar{m}}(\R^T_+)^*$ separating $K$ from $(C^b_{\bar{m}})_{++}(\R^T_+),$ i.e. some $\mu =\mu^r+\mu^s\in\mathcal{M}
(\check{\R^T_+})$ such that
\begin{equation}\label{A11}
\int\tfrac{f(x)}{\bar{m}(x)}\,d\mu(x)>0 \qquad\qquad\textrm{for all}\, f\in(C^b_{\bar{m}})_{++}(\R^T_+),
\end{equation}
while
\begin{equation}\label{A11a}
\int\tfrac{f(x)}{\bar{m}(x)}\,d\mu(x)\le0 \qquad\qquad\textrm{for all}\, f\in K.
\end{equation}
Clearly \eqref{A11} implies that $\mu=\mu^r+\mu^s$ is positive. We first observe that we have $\mu^r\ne 0.$ Indeed, supposing $\mu^r=0$, we find  
\begin{align*}
\int\tfrac{\varphi_{N+1}(x)}{\bar{m}(x)}\,d\mu(x) &=\int\tfrac{\varphi_{N+1}(x)}{\bar{m}(x)}\,d\mu^s(x)=\int 1 \,d\mu^s(x)=\|\mu^s\|>0
\end{align*}
and this is in contradiction to \eqref{A11a}.

We now claim that $\mu^r$ also separates $K$ from $(C^b_{\bar{m}})_{++}(\R^T_+).$ On the one hand, $\mu^r$ is a positive measure on $\R^T_+$. 
Hence \eqref{A11} still holds true, with $\mu$ replaced by $\mu^r$.
On the other hand, for each $1\le n\le N+1,$ we have
  $$
  \int\tfrac{\varphi_n(x)}{\bar{m}(x)}\,d\mu^r(x)\le \int\tfrac{\varphi_n(x)}{\bar{m}(x)}\,d\mu(x)\le 0.
  $$
The second inequality follows from \eqref{A11a}. For the first inequality it suffices to remark that
  $$\int\tfrac{\varphi_n(x)^-}{\bar{m}(x)}\,d\mu^s(x)=0, \qquad n=1,\dotsc ,N+1,$$
  by \eqref{cond2J}.
By normalizing $\mu^r$ to $\pi:=\tfrac{\mu^r}{\|\mu^r\|}$, we find a positive probability $\pi$ on $\R^T$ with $\int\tfrac{\varphi_n(x)}{\bar{m}(x)}\,d\pi(x)\le 0,$ for $n=1,\ldots,N+1$.
Now define $\hat\pi$ by
$$
\tfrac{\,d\hat\pi}{\,d\pi}=\tfrac{1}{\bar{m}}\left(\int\tfrac{1}{\bar{m}}\,d\pi\right)^{-1}.
$$
We have that $\hat\pi$ is a positive probability on $\R^T$ with $\int\varphi_n\,d\hat\pi\le 0,$ for $n=1,\ldots,N+1$, which shows that $\P_{(\varphi_i)_{i=1}^{N+1}}\ne\emptyset.$ 
\end{proof}  
The above proposition is the basis for the proof of the non-trivial part of Theorem \ref{MFTAP_intro}. In the course of the argument we also use the following characterization of martingale measures.
\begin{equation} \label{def:admissible_mart_measures2}
\M=\left\{\pi\in\P(\R_+^T): 
\begin{array}{c} \mbox{$S_t$ has finite first moment w.r.t.\ $\pi$, $t\leq T$}\\
 \int_{\R^T_+}(\Delta\sint x)_T\,d\pi(x)=0,\; \Delta\in\mathcal{C}_b ,
\end{array}
\right\},
\end{equation}
where $\Delta\in\mathcal{C}_b$ means that $\Delta_t(x_1,\dotsc,x_t)$ is continuous and bounded for all $t=0,\ldots,T-1$. The proof of \eqref{def:admissible_mart_measures2} is straightforward, see for instance \cite{BeHePe11}.

\begin{proof}[Proof of Theorem~\ref{MFTAP_intro}, $\eqref{it:NA}\Rightarrow\eqref{it:EMM}$]
In fact, we prove a stronger result. We show that $(i)^*\Rightarrow\eqref{it:EMM}$, where condition $(i)^*$ is defined as
\begin{itemize}
\item[$(i)^*$] There is no model-independent arbitrage such that $\Delta\in\h_g$ (see Definition~\ref{Gdef.adm}).
\end{itemize}
Recall that  $\varphi_0(x_1, \ldots, x_T)= g(x_T)$ and set 
$$\varphi_{-1}(x_1, \ldots, x_T):=- \sum_{t=1}^{T-1} g'(x_t)(x_T-x_t) + T g(x_T).$$
Note that since no arbitrage strategy can be constructed using the option $\varphi_0$, and since $g'(x_t), t<T$ are $g$-admissible trading strategies, it follows that no arbitrage strategy can be constructed with the help of $\varphi_{-1}$. 
We make the crucial observation that due to the convexity of $g$ we have $ m\leq  \varphi_{-1}$ (see \eqref{ConvexFHedge}). Moreover, if $\pi$ is a martingale measure,
then $\int_{\R^T_+} \varphi_{-1}\,d\pi = \int_{\R^T_+} T\varphi_0 \,d\pi$. 
Note that  $ \M_{(\varphi_i)_{i\in I}}= \M_{(\varphi_i)_{i\in I}, m}$ by Jensen's inequality.\footnote{For notational convenience we use the abbreviation $\M_{(\varphi_i)_{i\in I}, m}$ for 
$\M_{\{ \varphi_i \colon i\in I\} \cup \{ m \}}$.}
We will use a compactness argument to show that this set is not empty.

Assume that we are given finite families $F_1, F_2$, where 
$F_1 \subseteq I$ and
\begin{equation}\label{FamDef}
\{ \varphi_i\}_{i\in F_2}\subseteq \{\Delta_t(x_1, \ldots, x_t)(x_{t+1}-x_t): t< T, \Delta_t \in C_b(\R^t_+)\}.
\end{equation}
Then there exists no arbitrage, in the sense of Proposition \ref{Mprop:FTAP1}, for the family 
\begin{align}
 \{\varphi_i\}_{i\in F_1\cup F_2\cup\{0\}\cup\{-1\}}.
\end{align}
Since $m\leq \varphi_{-1}$ there is still no arbitrage opportunity if we replace $\varphi_{-1}$ by $m$.  Since the functions $\Delta_t$ in \eqref{FamDef} are taken to be continuous and bounded we may apply Proposition \ref{Mprop:FTAP1} to the family 
$$ \{\varphi_i\}_{i\in F_1\cup F_2 \cup\{0\}}$$
to obtain that 
$$\P_{\{\varphi_i\}_{i\in F_1\cup F_2 \cup\{0\}},m}\neq \emptyset.$$
Since 
$$
\M_{(\varphi_i)_{i\in I}}=\bigcap_{F_1,F_2} \P_{\{\varphi_i\}_{i\in F_1\cup F_2 \cup\{0\}},m},
$$ 
it remains to prove that $\P_{\{\varphi_i\}_{i\in F_1\cup F_2 \cup\{0\}},m}$ is compact.

\medskip 

  Step 1. Relative compactness.\\
We show that the set $\P_{\{\varphi_i\}_{i\in F_1\cup F_2 \cup\{0\}},m}$ is tight, hence relatively compact by Prokhorov's theorem. First we recall that $\lim_{\|x\|\to\infty}\tfrac{m(x)}{\|x\|}=\infty$ and that $\int_{\R^T_+} m \, \,d\pi \le 0$ for $\pi \in \P_{\{\varphi_i\}_{i\in F_1\cup F_2 \cup\{0\}},m}$. This implies $-\infty<-a:=\min m < 0$ and that
for all $\delta$ there is $k_\delta$ s.t. $m > \frac1\delta$ on $K_\delta^c$, where $K_\delta:=[0,k_\delta]^{T}$. Hence
\begin{equation}\label{in1}
\int_{K_\delta^c} m \,d\pi\geq\frac1\delta\pi\left(K_\delta^c\right).
\end{equation}
Furthermore
\begin{equation*}
0\geq \int_{\R^T_+} m \,d\pi=\int_{K_\delta} m\,d\pi+\int_{K_\delta^c} m\,d\pi\geq -a\pi(K_\delta)+\int_{K_\delta^c} m\,d\pi,
\end{equation*}
that is,
\begin{equation}\label{in2}
\int_{K_\delta^c} m\,d\pi\leq a\pi(K_\delta).
\end{equation}
Putting things together we obtain
\begin{equation*}
\pi(K_\delta^c)\leq \delta\int_{K_\delta^c} m\,d\pi\leq\delta a\pi(K_\delta)\leq\delta a.
\end{equation*}
This proves that for each fixed $\epsilon>0$ there is $k$ ($=k_\delta$ for $\delta=\epsilon/a$) such that $\pi\left(([0,k]^{T})^c\right)\leq\epsilon$ for all $\pi\in\P_{\{\varphi_i\}_{i\in F_1\cup F_2 \cup\{0\}},m}$. Hence, $\P_{\{\varphi_i\}_{i\in F_1\cup F_2 \cup\{0\}},m}$ is tight and thus relatively compact by Prokhorov's theorem.\\
   
   Step 2. Closedness.\\
   Let $\pi_n\in\P_{\{\varphi_i\}_{i\in F_1\cup F_2 \cup\{0\}},m}$ be such that $(\pi_n)$ converges weakly to $\tilde\pi$. We are going to prove that $\tilde\pi\in\P_{\{\varphi_i\}_{i\in F_1\cup F_2 \cup\{0\}},m}$. Since $\tilde\pi$ is clearly a probability measure, we only need to prove that $\tilde\pi$ satisfies the admissibility constraints:
\[
\int_{\R^T_+}\varphi\,\,d\tilde\pi\leq 0,\quad\quad \varphi\in\{m,\varphi_i \colon i\in F_1\cup F_2\cup\{0\}\}.
\]
We will consider separately the two integrals $\int_{\R^T_+}\varphi^+\,\,d\tilde\pi$ and $\int_{\R^T_+}\varphi^-\,\,d\tilde\pi$.

First of all, for each $\varphi\in\{m,\varphi_i \colon i\in F_1\cup F_2\cup\{0\}\}$ 
and for every $u \in [0,\infty)$ we have the basic inequality 
\begin{equation*}
\limsup_{n \to \infty} \int_{\R^T_+}\varphi^+\,\,d\pi_n \geq \limsup_{n \to \infty} \int_{\R^T_+}\varphi^+\wedge u\,\,d\pi_n,
\end{equation*}
where the l.h.s.\ is finite due to the first condition in \eqref{cond2} and the r.h.s.\ actually is a \emph{limit} by definition of weak convergence. Taking the limit $u \to \infty$ on both sides, we obtain 
\begin{equation} \label{pos_inequality}
\limsup_{n\to\infty}\int_{\R^T_+}\varphi^+\,\,d\pi_n \geq \lim_{u\to\infty}\lim_{n\to\infty}\int_{\R^T_+}\varphi^+\wedge u\,\,d\pi_n = \lim_{u\to\infty}\int_{\R^T_+}\varphi^+\wedge u\,\,d\tilde\pi=\int_{\R^T_+}\varphi^+\,\,d\tilde\pi,
\end{equation}
by weak convergence and by monotone convergence.

Furthermore, we will show that
\begin{equation}\label{ineqm}
\liminf_{n\to\infty}\int_{\R^T_+}\varphi^-\,\,d\pi_n \le \int_{\R^T_+}\varphi^-\,\,d\tilde\pi.
\end{equation}
Inequality \eqref{pos_inequality} and equation \eqref{ineqm} together then yield  
\[
\int_{\R^T_+}\varphi\,\,d\tilde\pi=\int_{\R^T_+}\varphi^+\,\,d\tilde\pi-\int_{\R^T_+}\varphi^-\,\,d\tilde\pi\leq \limsup_{n\to\infty}\int_{\R^T_+}\varphi^+\,\,d\pi_n - \liminf_{n\to\infty}\int_{\R^T_+}\varphi^-\,\,d\pi_n = \limsup_{n\to\infty}\int_{\R^T_+}\varphi\,\,d\pi_n\leq 0,
\]
as wanted.

In order to prove \eqref{ineqm} we will use the previous step, that is, for any fixed $\epsilon>0$ there is $k=k_\epsilon>0$ such that $\pi_n(K^c)\leq\epsilon$ for all $n\in\N$, where $K:=[0,k]^{T}$.
By weak convergence of measures we have
\begin{equation}
 \liminf_{n\to\infty}\int_{K}\varphi^-\,\,d\pi_n \le \limsup_{n\to\infty}\int_{K}\varphi^-\,\,d\pi_n \le \int_{K}\varphi^-\,\,d\tilde{\pi}.\label{negative_part_compact}
\end{equation}
Therefore, if $(k_\epsilon)_\epsilon$ is bounded, then we are done. We hence suppose that $k_\epsilon\to\infty$ as $\epsilon\to 0$.
Note that $\int_{\R^T_+}m\,\,d\pi_n\leq 0$ gives $\int_{\R^T_+}(m+a+1)\,\,d\pi_n\leq a+1$, which in turn implies $\int_{A}(m+a+1)\,\,d\pi_n\leq a+1$ for every $A\subseteq \R^T_+$, being $m+a+1$ non-negative (actually, $m+a+1\geq 1$).
Thus for $\varphi\in\{m,\varphi_i \colon i\in F_1\cup F_2\cup\{0\}\}$ we have
\[
a+1\geq\int_{K^c}(m+a+1)\mathbbm{1}_{\varphi^->0}\,\,d\pi_n\geq\int_{K^c}\varphi^-\min_{K^c}\frac{(m+a+1)}{\varphi^-}\mathbbm{1}_{\varphi^->0}\,\,d\pi_n,
\]
which implies
\[
\int_{K^c}\varphi^-\,\,d\pi_n=\int_{K^c}\varphi^-\mathbbm{1}_{\varphi^->0}\,\,d\pi_n\leq(a+1)\max_{K^c}\frac{\varphi^-}{(m+a+1)}.
\]
Now note that for all  $\varphi\in\{m,\varphi_i \colon i\in F_1\cup F_2\cup\{0\}\}$ we have that $\displaystyle{\max_{K^c}\frac{\varphi^-}{(m+a+1)}\to 0}$ as $\epsilon\to 0$. From this it follows that
\begin{equation}
\lim_{n\to\infty}\int_{K^c}\varphi^-\,\,d\pi_n \to 0 \label{negative_part_rest}
\end{equation}
as $\epsilon\to 0$, uniformly in $n$. Together, \eqref{negative_part_compact} and \eqref{negative_part_rest} imply that
\[
\liminf_{n\to\infty}\int_{\R^T_+}\varphi^-\,\,d\pi_n  \le  \int_{\R^T_+}\varphi^-\,\,d\tilde\pi,
\]
as claimed.
This concludes the proof.
\end{proof}

\begin{remark} 
If the stock prices process is not allowed to  take values on the (whole) half-line $\R_+$ but is restricted to a bounded interval $[0,b]$, the above considerations simplify significantly. In this case the path-space is compact, all continuous functions $\phi_i$ are bounded and the set of admissible measures is automatically compact; there is no need to require the existence of options whose payoff grows super-linearly. As a consequence, in this setting the robust FTAP follows in  a straightforward way from the Hahn-Banach Theorem.  
\end{remark}

\section{Super-Replication Theorem}\label{sect:srt}
The Super-Replication Theorem~\ref{M1thm:superrep} is a direct consequence of the Fundamental Theorem of Asset Pricing, Theorem~\ref{MFTAP_intro}.

\begin{proof}[Proof of Theorem~\ref{M1thm:superrep}]
By Remark \ref{rmk.js},
$\OM(\Phi)\leq   \OR(\Phi)$. It remains to prove the converse inequality.
In fact we prove a result which is stronger than the one stated. That is, we show this inequality when using only $g$-admissible strategies in the dual problem, i.e., when replacing $\h$ by $\h_g$ in the minimization problem in \eqref{DualRep}.
Let us first consider the case of continuous $\Phi$ satisfying \eqref{condwp} and
\begin{equation}\label{lowb}
\lim_{\|x\|\to\infty} \tfrac{\Phi(x)^-}{m(x)} < \infty.
\end{equation}
 Now suppose that the inequality is strict, that is, there exists $p$ such that
\begin{equation}\label{eq:p}
  \OM(\Phi)<p<  \OR(\Phi).
\end{equation}
Define $\varphi:=-\Phi+p$ and note that Theorem \ref{MFTAP_intro} applies to the set of constraints $\{\varphi,\varphi_i, i\in I\}$, implying the equivalence of the following:
\begin{itemize}
\item[(i)]\label{it:NA2} $\not\exists$ $f(x)=\sum_{n=1}^Na_n\varphi_{i_n}(x)+{\varphi}(x)+(\Delta\sint x)_T>0$ with $a_n\geq0$ and $\Delta\in\h_g$,
\item[(ii)]\label{it:EMM2}$\M_{(\varphi_i)_{i\in I},{\varphi}}\neq\emptyset$.
\end{itemize}
Therefore, either there exists $\pi\in\M_{(\varphi_i)_{i\in I}}$ such that
\begin{equation}\label{case1}
\int_{\R^T_+}\Phi \,d\pi\geq p,
\end{equation}
or there exist $a_n\geq0$ and $\Delta\in\h_g$ such that
\begin{equation}\label{case2}
p+\sum_{n=1}^Na_n\varphi_{i_n}(x)+(\Delta\sint x)_T>\Phi(x).
\end{equation} 
Note that \eqref{case1} would imply $  \OM(\Phi)\geq p$, in contradiction to the first inequality in \eqref{eq:p}, and that \eqref{case2} would imply $  \OR(\Phi)\leq p$, in contradiction to the second inequality in \eqref{eq:p}. This shows that there is no $p$ as in \eqref{eq:p}, hence the duality stated in the theorem holds for all continuous $\Phi$ which satisfy \eqref{condwp} and \eqref{lowb}.
Now note that any u.s.c.\ function $\Phi$ satisfying \eqref{condwp} can be written as an infimum over continuous functions $\Phi_n,n\in\N$ satisfying \eqref{condwp} and \eqref{lowb}. 
By a standard argument, the duality relation then carries over from $\Phi_n$ to $\Phi$. This is worked out in detail for instance in \cite[Proof of Thm.\ 1]{BeHePe11} in a very similar setup. 
At the same place the reader can find the argument showing that the supremum in \eqref{PrimalRep} is attained. 
\end{proof}

\section{Ramifications of the Super-Replication result}\label{LastOne}
We start with a corollary of the previous results which avoids the asymmetry present in the requirements on $\varphi_0$.
To achieve this, we assume that there exists a sequence of call options written on $S_T$ whose strikes $K_n$ tend to $\infty$. We call $p_n$ the corresponding market prices and use the notation
\begin{equation}\label{eq.psin}
\psi_n(y):=(y-{K_n})_+,\qquad \tilde \psi_n(x):= (\psi_n(x_T)-p_n),\quad n\geq 1.
\end{equation}
\begin{assumption}\label{ass}
Let $\varphi_i:\R_+^T\to\R, i\in I$ be continuous functions including $\tilde \psi_n, n\geq 1$, and assume that $\M_{(\varphi_i)_{i\in I}} \neq \emptyset$.
Let $\alpha_n\geq 0$ be such that $\sum_{n=1}^\infty\alpha_n=\infty$ and $\sum_{n=1}^\infty\alpha_n p_n < \infty$. We set 
\[g_0(y):=\sum_{n=1}^\infty\alpha_n(\psi_n(y)- p_n),\quad m_0(x):=\sum_{t=1}^Tg_0(x_t)\]
 and assume that for all $i\in I$
\begin{equation*}
\lim_{\|x\|\to\infty} \tfrac{\varphi_i(x)^-}{m_0(x)}=0,\qquad
\lim_{\|x\|\to\infty} \tfrac{\varphi_i(x)^+}{m_0(x)}<\infty.
\end{equation*}
\end{assumption}
 Theorem~\ref{M1thm:superrep} can then be applied by setting $g=g_0$. Indeed, since $\tilde \psi_n, n\geq 1$ are already present in the admissibility resp.\ the super-replication condition, it makes no difference whether or not one includes also $g_0$ among the $\varphi_i$.
Hence we obtain:
\begin{corollary}\label{M2thm:superrep}
Let $(\varphi_i)_{i\in I}$ and $(\alpha_n)_{n\ge 1}$ be like in Assumption~\ref{ass}. Let $\Phi:\R_+^T\to\R$ be u.s.c.\ and assume that $\lim_{\|x\|\to\infty} \tfrac{\Phi(x)^+}{m_0(x)}=0$. Then
\begin{align*}
  \OM(\Phi):=&\textstyle\sup_{\pi\in\M_{(\varphi_i)_{i\in I}}}\int_{\R^T_+}\Phi(x)\,d\pi(x)\\
=&\ \ \ \inf\Big\{d: \textstyle \exists a_1,\ldots,a_N\geq 0, i_1,\ldots,i_N\in I, b_n\geq 0,
\sup_n b_n\alpha_n<\infty, \Delta\in \h\, \textrm{ s.t. }\\
& \textstyle \qquad\qquad d+\sum_{k=1}^Na_k\varphi_{i_k}+\sum_{n=1}^\infty b_n\alpha_n \tilde \psi_n +(\Delta \sint x)_T \geq\Phi\Big\}=:  \OR(\Phi).
\end{align*}
In addition, the above supremum is a maximum.
\end{corollary}

Note that this result can be easily put into a ``symmetric form'' including also $-\tilde \psi_n, n\geq 1$ in the family $(\varphi_i)_{i\in I}$.
Moreover, as a consequence of Corollary~\ref{M2thm:superrep}, we have the super-replication result under the assumption that the distribution $\nu$ of the asset at the terminal date $T$ is known. Here $I$ is simply taken to be the empty set, in which case we obtain exactly Corollary~\ref{thm:superrep_term_marg}.

We note that, for any convex super-linear function $\bar g:\R_+\to\R$ such that $\int_{\R_+} \bar g \,d\nu<\infty$, there exist constants $c$, $\alpha_n\geq 0$ and $K_n\nearrow\infty$ such that
\begin{equation}\label{eq.gak}
\bar g(y)\leq c+\sum_{n=1}^\infty\alpha_n(y-K_n)_+,\qquad \sum_{n=1}^\infty\alpha_n=\infty,\qquad \sum_{n=1}^\infty\alpha_np_n<\infty,
\end{equation}
where $p_n:=\int_{K_n}^\infty(y-K_n)\,d\nu(y)$.
Setting
\[\M_{(\varphi_i)_{i\in I}}(\nu):=\M_{(\varphi_i)_{i\in I}}\cap\M(\nu),\]
we obtain the following result.
\begin{corollary}\label{thm.nui}
Let $\varphi_i:\R_+^T\to\R, i\in I$ be continuous and growing at most linearly at infinity and assume $\M_{(\varphi_i)_{i\in I}}(\nu)\neq \emptyset$. For $\Phi:\R^T_+\to\R$ u.s.c.\ and linearly bounded from above we have
\begin{align}\label{eq.sup}
\OM(\Phi) :=& \sup_{\M_{(\varphi_i)_{i\in I}}(\nu)}\left\{\textstyle \int_{\R^T_+} \Phi(x)\,d\pi(x)\right\}\\
=&\ \  \inf\ \ \left\{\textstyle \int_{\R_+}\varphi(y)\,d\nu(y): \!\!
\begin{array}{l}
\varphi\in L^1(\nu), \exists\, \Delta\in\h, a_1,\ldots,a_N\geq 0, i_1,\ldots,i_N\in I, \\ 
\mbox{s.t. }\ \ 
\varphi(x_T)+\sum_{n=1}^N a_n\varphi_{i_n}(x)+(\Delta\sint x)_T\geq \Phi(x) 
\end{array}
\right\}=:\OR(\Phi).\nonumber
\end{align}
In addition, the above supremum is a maximum.

\smallskip

More generally these results hold true for $\phi_i, i\in I$ continuous and $\Phi$ u.s.c.\ if there exists a convex super-linear function $\tilde g:\R_+\to\R$ in $L^1(\nu)$ such that
\begin{equation}\label{eq.cond}
\lim_{\|x\|\to\infty} \tfrac{|\varphi_i(x)|}{\sum_{t=1}^T\tilde g(x_t)}<\infty,\quad \lim_{\|x\|\to\infty} \tfrac{\Phi(x)^+}{\sum_{t=1}^T\tilde g(x_t)}<\infty.
\end{equation}
\end{corollary}

\begin{proof}
Step 1. Let $\phi_i, i\in I$ be continuous and such that $\M_{(\varphi_i)_{i\in I}}(\nu)\neq \emptyset$, and $\Phi$ be u.s.c.\ and such that \eqref{eq.cond} holds for some convex super-linear function $\tilde g\in L^1(\nu)$.
By applying Lemma~\ref{lemma.int} to $f=\tilde g$, we obtain a convex super-linear function $\bar g$ in $L^1(\nu)$ such that
\[
\lim_{\|x\|\to\infty} \tfrac{\varphi_i(x)^-}{\sum_{t=1}^T\bar g(x_t)}=0,\quad \lim_{\|x\|\to\infty} \tfrac{\varphi_i(x)^+}{\sum_{t=1}^T\bar g(x_t)}=0,\quad \lim_{\|x\|\to\infty} \tfrac{\Phi(x)^+}{\sum_{t=1}^T\bar g(x_t)}=0.
\]
Now consider $\alpha_n\geq 0$ and $K_n\nearrow \infty$ as in \eqref{eq.gak}. We can include the corresponding functions $\tilde\psi_n$ defined as in \eqref{eq.psin} among the $\varphi_i$ since this neither changes the set of admissible martingale measures nor introduces arbitrage. Now, applying Corollary~\ref{M2thm:superrep} we obtain
\begin{align*}
\sup_{\M_{(\varphi_i)_{i\in I}}(\nu)}\left\{\textstyle \int_{\R^T_+} \Phi(x)\,d\pi(x)\right\} =& \inf\left\{\textstyle d: \!\!
\begin{array}{l}
\exists\, \Delta\in\h, a_1,\ldots,a_N\geq 0, i_1,\ldots,i_N\in I, b_n\geq 0,
\sup_n b_n\alpha_n<\infty\\ 
\mbox{s.t.}\ \ 
d+\sum_{n=1}^N a_n\varphi_{i_n}(x)+\sum_{n=1}^\infty b_n\alpha_n \tilde \psi_n+(\Delta\sint x)_T\geq \Phi(x) 
\end{array}
\right\}\nonumber\\[0.1cm]
\geq& \inf\left\{\textstyle \int_{\R_+}\varphi(y)\,d\nu(y): \!\!
\begin{array}{l}
\varphi\in L^1(\nu), \exists\, \Delta\in\h, a_1,\ldots,a_N\geq 0, i_1,\ldots,i_N\in I\\ 
\mbox{s.t.}\ \ 
\varphi(x_T)+\sum_{n=1}^N a_n\varphi_{i_n}(x)+(\Delta\sint x)_T\geq \Phi(x) 
\end{array}
\right\}.
\end{align*}
This gives $\OM(\Phi)\geq\OR(\Phi)$, hence $\OM(\Phi)=\OR(\Phi)$, the other inequality being trivial.
The fact that the supremum in \eqref{eq.sup} is attained is again obtained by standard arguments, cf. \cite[Theorem 1]{BeHePe11}.

Step 2. Let now $\varphi_i, i\in I$ be continuous and growing at most linearly at infinity and $\Phi$ be u.s.c.\ and linearly bounded from above.
By applying Lemma~\ref{lemma.int} to $f(x)=|x|$, we obtain a convex super-linear function $\bar f$ in $L^1(\nu)$ such that \eqref{eq.cond} is satisfied with $\bar g=\bar f$. Now we can apply Step 1, which concludes the proof.
\end{proof}

The following Lemma, used in the proof of Corollary \ref{thm.nui}, is a rather simple consequence of the de la Vall\'{e}e-Poussin Theorem.
\begin{lemma}\label{lemma.int}
Let $\mu$ be a probability measure on $\R_+$ having finite first moment and let $f:\R_+ \to\R$ be a convex function in $L^1(\mu)$. Then there exists a convex function $\bar f:\R_+\to\R$ in $L^1(\mu)$ such that $\frac{|\bar f(x)|}{|f(x)|}\to\infty$ as $x\to\infty$.
\end{lemma}

It seems natural to assume that the market does not only yield information about the call options at the terminal time $T$. In fact,  in \cite{BeHePe11} a super-replication result is proved for the case where all marginals $S_t \sim \nu_t, t=1,\ldots, T$ are known. By Theorem 1 in \cite{BeHePe11} we have: 

\begin{corollary}\label{BHP}
Assume that $\nu_t, t=1, \ldots, T$ are probability measures on $\R_+$ with barycenter $S_0$ such that the set $\mathcal{M}(\nu_1,\ldots, \nu_T)$ of martingale measures $\pi$ satisfying $S_t(\pi)=\nu_t$ is non-empty. 
Let $\Phi:\R^T_+\to\R$ be u.s.c.\ and linearly bounded from above.
Then
\begin{align*}
\OM(\Phi) :=&\ \sup_{\pi\in\mathcal{M}(\nu_1, \ldots, \nu_T)}\left\{\textstyle \int_{\R^T_+} \Phi(x)\,d\pi(x)\right\}\\
=&\ \ \ \inf\ \ \Big\{\textstyle \sum_{t=1}^T\int_{\R_+}\varphi_t\,d\nu_t:\varphi_t\in L^1(\nu_t), \exists\, \Delta\in\h 
\ \textrm{s.t.}\ 
\sum_{t=1}^T\phi_t(x_t)+(\Delta\sint x)_T\geq \Phi(x) \Big\}=:\OR(\Phi).
\end{align*}
In addition, the above supremum is a maximum.
\end{corollary}

This follows precisely in the same way as Corollary~\ref{thm.nui}, by   including the options $\{\pm\tilde\psi_{k,t}, k\in\R_+, t=1,\ldots,T-1\}$ among the $(\varphi_i)_{i\in I}$, where
\[\tilde\psi_{k,t}(x):=(x_t-k)_+-\int_k^\infty(y-k)\,d\mu_t(y).\]

\section{Connection with Martingale-Inequalities}\label{SSSS}
In this section we illustrate how the Super-Replication Theorem~\ref{M1thm:superrep}  connects to the field of martingale inequalities. 
We will concentrate on  the particular case of the Doob-$L^1$ inequality. In its sharp version obtained by Gilat \cite{Gi86} it asserts that for every non-negative martingale $S=(S_t)_{t=0}^T$ starting at $S_0=1$ we have
\begin{align}\label{LogDoob}
  \E [\bar S_T]&\leq\frac{e}{e-1}\Big[\ \E[S_T\log(S_T)]+1\Big],
      \end{align}
where $\bar S_T$ is the supremum of $S$ up to time $T$.

Having Theorem~\ref{M1thm:superrep} in mind, it is natural to ask whether there exists a path-wise hedging inequality associated to it. This is indeed the case.
 
\begin{claim} Fix $C\geq 0$. For every $\eps>0$ there exist $a\geq 0$ and $\Delta$ such that
 \begin{align}\label{DoobSemiA}
\bar x_T \leq a \left( x_T\log (x_T) -  C \right)  + \frac{e}{e-1} (C+1) + \eps   +  (\Delta \sint x)_T
\end{align}
 for all $x_0, x_1, \ldots, x_T\in \R_+$.
\end{claim}
\begin{proof}
Fix $C$ and $\eps$. To establish a connection with the robust Super-Replication Theorem, we let $x_0:=1$ and interpret $\Phi(x_1, \ldots, x_T):=\bar x_T=\max(x_0,\ldots,x_T)$ and $\varphi(x_T):=x_T\log(x_T)$ as financial 
derivatives, where $\varphi$ can be bought at price $ C $ on the market. Our task is then to determine a reasonable upper bound for the price of $\Phi$. (Note that from $C\geq 0$ it follows that the set of admissible martingale measures is non-empty as witnessed by the constant process $S\equiv 1$.) By  \eqref{LogDoob} we have
\begin{align}\label{PrimalDoob}
\sup\left\{\int_{\R^T_+} \bar x_T \,d\pi: \pi\in \M, \int_{\R^T_+} x_T \log(x_T) \,d\pi
\leq C \right\} &\leq \frac{e}{e-1} (C+1),
\end{align} 
where $\M$ is the set of all martingale measures.
Applying  Theorem~\ref{M1thm:superrep} to $\Phi$ and $\varphi_0:=\varphi-C$ we thus obtain that $\bar S_T$ can be super-replicated path-wise  using an initial endowment of at  most $\frac{e}{e-1} (C+1)+\eps$.
 This is precisely what is asserted in \eqref{DoobSemiA}.
\end{proof}

These considerations provided the motivation to search for an explicit super-replication strategy for $\Phi(x)=\bar x_T$ (see \cite{AcBePeScTe12}). Indeed \eqref{DoobSemiA} holds (independent of $C$) for the particular choices  
  $a=\frac{e}{e-1}, \eps=0$ and $\Delta_t(x_1, \ldots,x_t)= -\log (\bar x_t)$, where it  corresponds to
\begin{align}
\bar x_T \leq   \frac{e}{e-1} \left(x_T\log (x_T) + 1 \right) - \left(\log (\bar x_t) \sint x \right)_T. \label{DoobSemi}
\end{align}
Let us stress that \eqref{DoobSemi} is simply an inequality for non-negative \emph{numbers} $x_1,\ldots, x_T$. Its verification, using convexity of $x\mapsto x \log(x)$, is entirely elementary (\cite[Proposition 2.1]{AcBePeScTe12}).  An application of  \eqref{DoobSemi} is that it implies Doob's $L^1$-inequality:
\begin{proof}[Proof of \eqref{LogDoob}.] Apply \eqref{DoobSemi} to the paths of $(S_n)_{n=0}^T $ and take expectation to obtain 
\begin{align*}
\E[ \bar S_T] \leq\ & \frac{e}{e-1}\big[\E[ S_T\log (S_T)] + 1\big] - \E[ (\log (\bar S_t) \sint S)_T] \\ =\ &\frac{e}{e-1}\big[\E[ S_T\log (S_T)] + 1\big].\qedhere
\end{align*}
\end{proof}
We emphasize that by Theorem~\ref{M1thm:superrep} one knows \emph{a priori} that a  path-wise hedging strategy exists and hence that the Doob $L^1$-inequality can be proved \emph{in this way}. 
In particular one expects that the same strategy of proof can be applied  to a variety of other inequalities.

\bibliography{joint_biblio}{}

\bibliographystyle{alpha}

\end{document}